\documentclass[12pt,espf]{amsart}
\usepackage{a4}
\usepackage{amssymb}
\usepackage{amsmath}
\usepackage{amsthm}
\usepackage{amstext}
\usepackage{amscd}
\usepackage{latexsym}
\usepackage{graphics}
\usepackage{color}

\theoremstyle{plain}
\newtheorem{thm}{Theorem}[section]

\theoremstyle{definition}

\newtheorem{Definition-Proposition}[thm]{Definition-Proposition}

\renewcommand{\tilde}{\widetilde}

\newcommand{\Q}{{\mathbb Q}}

\input{xy}
\xyoption{all}

\begin{document}
\title{On Log Canonical Inversion of Adjunction}
\author{Christopher D. Hacon}
\date{\today}
\address{Department of Mathematics \\
University of Utah\\
155 South 1400 East\\
JWB 233\\
Salt Lake City, UT 84112, USA}
\email{hacon@math.utah.edu}
\thanks{The author was partially supported by NSF research grant no: 0757897}
\maketitle
\begin{center}{\it To V. Shokurov on his sixtieth birthday}\end{center}
\begin{abstract} We prove a result on the inversion of adjunction for log canonical pairs that generalizes Kawakita's result to log canonical centers of arbitrary codimension.
\end{abstract}

The minimal model program is an ambitious program that aims to generalize to higher dimensional varieties many of the results in the classification of surfaces obtained by the Italian school of Algebraic Geometry in the early 20-th century. 
Log canonical singularities are the largest class of singularities for which the minimal model program is expected to hold. Let $(X,\Delta )$ be a {\it pair} consisting of a normal variety $X$ and an effective $\Q$-divisor $\Delta =\sum \delta _i\Delta _i$ such that $K_X+\Delta$ is $\Q$-Cartier. Consider a {\it log resolution} of $(X,\Delta )$ i.e. a proper birational morphism $\mu :Y\to X$ such that $Y$ is smooth, ${\rm Exc }(\mu)$ is a divisor and $\mu ^{-1}_* \Delta +{\rm Exc}(\mu )$ has simple normal crossings. Write $K_Y+\Delta _Y=\mu ^*(K_X+\Delta )$, then $\Delta _Y=\sum b_iB_i$ is uniquely determined and $(X,\Delta )$ is {\it log canonical} (resp. {\it klt}) if $b_i\leq 1$ (resp. $b_i<1$) for all $i$.
Similarly $(X,\Delta )$ is {\it plt} if $0\leq \delta _i\leq 1$ and for all log resolutions $\mu:Y\to X$ we have $b_i<1$ for all $i$ such that $B_i$ is $\mu$-exceptional.
Suppose that $\Delta =S+B$ where $S$ is a prime divisor and $\nu :S^\nu \to S$ is the normalization, then there is a uniquely defined $\Q$-divisor 
${\rm Diff}(B)$ on $S^\nu$ such that $(K_X+S+B)|_{S^\nu}=K_{S^\nu}+{\rm Diff}(B)$ (cf. Section 16 of \cite{Kollar93}). Note that if $S'$ is the strict transform of $S$ on $Y$ and $\bar \mu :S'\to S^\nu$ is the induced morphism, then ${\rm Diff}(B)=\bar \mu _* ((\Delta _Y-S')|_{S'})$.

It is of fundamental importance in the minimal model program to relate the singularities of the pair $(X,S+B)$ to those of the pair $(S^\nu , {\rm Diff}(B))$. If we know that the pair $(X,S+B)$ is log canonical (resp. plt), then it is easy to see that the pair   $(S^\nu , {\rm Diff}(B))$ is also log canonical (resp. klt). This process is known as adjunction.
Inverse of adjunction on the other hand is the process of deducing information about the singularities on the ambient variety $(X,S+B)$ from information on the singularities of the divisor $(S^\nu , {\rm Diff}(B))$.
These results are much more subtle and useful.
It is known that if $(S^\nu , {\rm Diff}(B))$ is klt, then $(X,S+B)$ is plt on a neighborhood of $S$ (cf. \cite[17.6]{Kollar93}) and by a more recent result of Kawakita, that if $(S^\nu , {\rm Diff}(B))$ is lc, then $(X,S+B)$ is lc on a neighborhood of $S$ (cf. \cite{Kawakita07}).
The proof of both results heavily relies on Kawamata-Viehweg vanishing. 

The purpose of this short note is to give a proof of a generalization of Kawakita's theorem on the inversion of adjunction to higher co-dimensional subvarieties. Our proof is based on the results of \cite{BCHM10} and recovers a new proof of Kawakita's theorem. Our argument closely follows ideas of Shokurov (cf. \cite{Shokurov93}), but avoids the use of the minimal model program for log canonical pairs. Note moreover that there are some similarities with this proof and the arguments of \cite{Kawakita07} and \cite{Kollar93}.
See also \cite{Schwede09} for a related result in characteristic $p>0$.

Before stating the main theorem we must introduce some notation.
Let $(X,\Delta )$ be a pair, then a subvariety $V\subset X$ is a {\it log canonical center} if there is a log resolution $\mu :Y\to X$ such that $K_Y+\Delta _Y=\mu ^*(K_X+\Delta _X)$ where $\Delta _Y=\sum b_iB_i$ and 
${\rm max}\{ b_i|\mu (B_i)=V\}=1$. Recall that in this case, $(X,\Delta )$ is log canonical on a neighborhood of the generic point of $V$ (cf. \cite[17.1.1]{Kollar93}). We will denote by $S$ a component $B_i$ as above, such that $b_i=1$ and $\mu (B_i)=V$. 
Let $\nu :W\to V$ be a birational morphism from a normal variety $W$,
$\Delta _S:=(\Delta _Y-S)|_S$,  and assume that $\bar \mu :S \to W$ is a morphism.
Then we define an effective $\Q$-divisor $B _{W}=\sum (1-t_i)P_i$ on $W$ as follows:
for any codimension $1$ point $P_i$ on $W$, let $$t_i={\rm sup}\{\tau |(S,\Delta _S+\tau \bar \mu ^*P_i){\rm \ is\ lc\ over\ }\eta _{P_i}\}$$ where $\eta _{P_i}$ denotes the generic point of $P_i$.  Note that the $t_i$ are  rational numbers (positive or negative).  
It is known that (cf. \cite{Ambro99}) \begin{enumerate}
\item The numbers $t_i$ are independent of the log resolution $\mu :Y\to X$ and of the choice of the divisor $S$.
\item If $W=V^\nu$ is the normalization of $V$, then $1-t_i\geq 0$ for all $P_i\in V^\nu$, and the strict inequality only holds for finitely many codimension $1$ points $P_i\in W$. 
\item If $S$ is the only component of $\Delta _Y$ of coefficient $1$, then $(W,B _{W })$ is klt. 
\item When $\dim V=\dim X-1$, we let $S$ be the strict transform of $V$ and we have $B _{V^\nu}={\rm Diff} (\Delta -V)$ where $V^\nu \to V$ is the normalization.
\item If $\eta : W'\to W$ is a birational morphism of normal varieties, then $\eta _* B _{W'}=B _W$ so that we have a {\bf b}-divisor ${\bf B}={\bf B}(V;X,\Delta)$ defined by ${\bf B}_W=B _W$ (see \cite{Corti07} for the definition and properties of {\bf b}-divisors).
\item If $S\to W$ satisfies the standard normal crossing assumptions of \cite[8.3.6]{Kollar07}, then  ${\bf B}$ descends to $W$ in the sense that for any birational morphism $\eta:W'\to W$, we have  $\eta ^*(K_W+{\mathbf B}_W)=K_{W'}+{\mathbf B}_{W'}$. 
\item If $(W, B _W)$ is sub-log canonical (i.e. if $t_i\geq 0$ for any $P_i\in W$) for any sufficiently high model (or equivalently for any model $W$ such that $S\to W$ satisfies the standard normal crossing assumptions) then we say that $(V^\nu, {\bf B})$ is log canonical. Note that by the Base Change Conjecture (cf. \cite[Conjecture 5]{Ambro99}) it is expected that $K_{V^\nu}+B_{V^\nu}$ is $\Q$-Cartier and  $ {\bf B}$ descends to $V^\nu$ in which case $(V^\nu, {\bf B})$ is log canonical if and only if $(V^\nu ,B _{V^\nu})$ is log canonical.
\end{enumerate}
We will prove the following generalization of Kawakita's result:

\begin{thm} Let $V$ be a log canonical center of a pair $(X,\Delta =\sum \delta _i\Delta _i)$ where $0\leq \delta _i \leq 1$. Then $(X,\Delta )$ is log canonical on a neighborhood of $V$ if and only if $(V^\nu,{\bf B}(V;X,\Delta))$ is log canonical.
\end{thm}
\begin{proof} If $(X,\Delta )$ is log canonical on a neighborhood of $V$, then it is well known and easy to see that $(V^\nu,{\bf B}(V;X,\Delta))$ is log canonical. 
Therefore, we will assume that $(V^\nu,{\bf B}(V;X,\Delta))$ is log canonical and we will prove that $(X,\Delta )$ is log canonical on a neighborhood of $V$. 
Let $\mu :Y\to X$ be a dlt model (cf. \cite[3.1]{KK10}) of $(X,\Delta )$, so that if we write
$\mu ^*(K_X+\Delta)=K_Y+\Delta _Y$, then
\begin{enumerate}
\item $Y$ is $\Q$-factorial, 
\item $\Delta _Y=\sum b_iB_i\geq 0$, 
\item $(Y,\Delta _Y'=\sum _{b_i\leq 1}b_iB_i+\sum _{b_i>1}B_i)$ is dlt, and 
\item every exceptional divisor appears in $\Delta _Y'$ with coefficient $\geq 1$. \end{enumerate}
We may assume that $\Delta _Y'=S+\Gamma$ where $S$ is an irreducible component of $\Delta _Y^{=1}=\sum _{b_i= 1}B_i$ that dominates $V$ and $\Sigma =\Delta _Y-S-\Gamma$.
Note that $f(\Sigma )\not \supset V$. 
We now fix a sufficiently ample divisor $H$ on $Y$ and we run the $(K_Y+S+\Gamma)$-MMP with scaling of $H$ over $X$ (cf. \cite[3.10]{BCHM10}).
Let $\phi _i:Y_i\dasharrow Y_{i+1}$ be the induced sequence of flips and divisorial contractions and let $\mu _i :Y_i\to X$ be the induced morphisms.
For any divisor $G$ on $Y$ we let $G_i$ be its strict transform on $Y_i$.
Then there is a non-increasing sequence of rational numbers $s_i\geq s_{i+1}$ which is either \begin{itemize}\item finite with $s_{N+1}=0$, or 
\item infinite with $\lim s_i=0$\end{itemize} such that 
$K_{Y_i}+S_i+\Gamma _i+sH_i$ is nef over $X$ for all $s_i\geq s\geq s_{i+1}$.
For $i\geq i_0$, we may assume that each $\phi _i$ is a flip.
By a well known discrepancy computation, we may also assume that $S_i\dasharrow S_{i+1}$ is an isomorphism in codimension $1$ for all $i\geq i_0$ (cf. the arguments in the proofs of Step 1 and Step 2 of \cite[4.2.1]{Fujino07}).
For any $t>0$ there exists a $\Q$-divisor $\Theta _t$ on $Y$ such that $\Theta _t\sim _{\Q}\Gamma +tH$ and 
$(Y,S+\Theta _t )$ is plt. Note that if $t<s_i$, then $({Y_i}, S_i+\Theta _{t,i})$ is plt (this is because plt singularities are preserved by steps of the minimal model program). 

Suppose that for some $i\geq 0$, we have $S_i\cap \Sigma _i\ne \emptyset$, then $$(\mu _i^*(K_X+\Delta ))|_{S_i}=(K_{Y_i}+S_i+\Gamma _i+\Sigma _i)|_{S_i}=K_{S_i}+{\rm Diff}(\Gamma _i+\Sigma _i)$$ 
is not log canonical. Let $\bar \mu _i :S_i\to V^\nu$ be the induced morphism.
We may replace $S_i\to V^\nu$ by a birational model $\tilde \mu :\tilde S\to W$
satisfying the standard normal crossing assumptions.
If $g:\tilde S\to S_i$ is the induced morphism and we write $K_{\tilde S}+\Delta _{\tilde S}=g^*(K_{S_i}+{\rm Diff}(\Gamma _i+\Sigma _i))$, then there is a component of $\Delta _{\tilde S}$ of coefficient $>1$.
After possibly replacing $W$ by an appropriate birational model, we may assume that the image of this component is a codimension $1$ point $P_i\in W$. But then $t_i<0$ so that $1-t_i>1$ and hence  
$(W,B _{W})$ is not log canonical. This proves the theorem. 

Therefore, we will assume that $S_i\cap \Sigma _i= \emptyset$ for all $i\ge  0$, and we will derive a contradiction. Note that if this is the case, then any curve contained in $S_i$ intersects $\Sigma _i$ trivially and hence $\phi _i$ does not contract $S_i$.
For any $m\gg 0$ such that $m\Sigma$ is an integral divisor, let $i\gg 0$ be the integer such that
$s_i> \frac 1 m \geq s_{i+1}$.
Notice that since $$H_i-m\Sigma _i-S_i\sim _{\Q, X}K_{Y_i}+\Theta _{\frac 1 m ,i}+(m-1)(K_{Y_i}+S_i+\Gamma _i+\frac 1 m H_i)$$ where $( {Y_i},\Theta _{\frac 1 m ,i})$ is klt and $K_{Y_i}+S_i+\Gamma _i+\frac 1 m H_i$ is nef over $X$, then by Kawamata-Viehweg vanishing (cf. \cite[2.70]{KM98}), we have that
$R^1 (\mu _{i})_*\mathcal O _{Y_{i}}(H_{i}-m\Sigma_i -S_i)=0$ and hence
$$(\mu _{i})_*\mathcal O _{Y_{i}}(H_{i}-m\Sigma _i)\to ({\bar \mu }_{i})_*\mathcal O _{S_{i}}(H_{i})$$
is surjective.
On the other hand, for $m\gg 0$, the subsheaves
$$(\mu _{i})_*\mathcal O _{Y_{i}}(H_{i}-m\Sigma _i)=(\mu _{i_0})_*\mathcal O _{Y_{i_0}}(H_{i_0}-m\Sigma _{i_0})\subset (\mu _{i_0})_*\mathcal O _{Y_{i_0}}(H_{i_0})$$ 
are contained in $\mathcal I _{\mu _{i_0}(\Sigma _{i_0})}\cdot  (\mu _{i_0})_*\mathcal O _{Y_{i_0}}(H_{i_0})$. Since $V\cap \mu _{i_0}(\Sigma_{i_0})\ne \emptyset$ and $S_{i_0}\dasharrow S_i$ is an ismorphism in codimension $1$, the induced homomorphism
$$(\mu _{i})_*\mathcal O _{Y_{i}}(H_{i}-m\Sigma _i)\to ({\bar \mu }_{i_0})_*\mathcal O _{S_{i_0}}(H_{i_0})= ({\bar \mu }_{i})_*\mathcal O _{S_{i}}(H_{i})$$ is not surjective. This is the required contradiction.

\end{proof}

 \end{document}